\documentclass[11pt,reqno,oneside]{amsart}

\usepackage{a4wide,amsthm,amsfonts}
\usepackage{latexsym}
\usepackage{amsmath}
\usepackage{amssymb}
\usepackage[numbers]{natbib}

\theoremstyle{plain}
\newtheorem{Thm}{Theorem}[section]
\newtheorem{Veta}[Thm]{Theorem}
\newtheorem{Lemma}[Thm]{Lemma}
\newtheorem{Tvrzeni}[Thm]{Proposition}
\newtheorem{Dusledek}[Thm]{Corollary}
\newtheorem{Definice}[Thm]{Definition}
\theoremstyle{definition}
\newtheorem{Priklad}[Thm]{Example}
\theoremstyle{remark}
\newtheorem{remark}[Thm]{Remark}

\numberwithin{equation}{section}

\renewcommand{\P}{\ensuremath{\mathbb{P}\,}}
\newcommand{\E}{\ensuremath{\mathbb{E}\,}}
\newcommand{\Ref}[1]{(\ref{#1})}
\newcommand{\R}{\ensuremath{\mathbb{R}}}
\newcommand{\N}{\ensuremath{\mathbb{N}}}

\renewcommand{\d}{\ensuremath{{\rm d}}}

\newcommand{\konv}[1]{\mbox{${} \atop \stackrel{\relbar\joinrel\mkern-2mu \relbar\joinrel\mkern-2mu \relbar\joinrel\mkern-2mu
                      \relbar\joinrel\mkern-2mu \relbar\joinrel\mkern-2mu \relbar\joinrel\mkern-2mu \relbar\joinrel\mkern-2mu
                      \longrightarrow}{\scriptstyle{#1 \rightarrow+\infty}}$} }

\begin{document}
\title[]
{STOCHASTIC AFFINE EVOLUTION EQUATIONS WITH MULTIPLICATIVE FRACTIONAL NOISE}
\author{B.~Maslowski}
\address{
Faculty of Mathematics and Physics, Charles University in Prague,
Sokolovsk\'{a} 83, 186 75 Prague 8, Czech Republic}
\email{maslow@karlin.mff.cuni.cz}
\author{J.~\v{S}nup\'{a}rkov\'{a}}
\address{Department of Mathematics, Faculty of Chemical Engineering, University of Chemical Technology Prague,
Studentsk\'{a} 6, 166 28 Prague 6 -- Dejvice, Czech Republic}
\email{snuparkj@vscht.cz}
\thanks{
\indent This work was partially supported by the GA\v{C}R grant no.
15-08819S.}

\begin{abstract}
{\it A stochastic affine evolution equation with bilinear noise term is studied where the
driving process is a real--valued fractional Brownian motion. Stochastic integration is understood in the Skorokhod sense. Existence and uniqueness of weak solution is proved and some results on the large time dynamics are obtained. }
\end{abstract}

\maketitle

\section{Introduction}

In the paper the formula for stochastic evolution system
generated by equation with bilinear stochastic term and affine drift term is studied. Existence and uniqueness of solutions is proved and the relation between weak and ``mild'' form of solutions is investigated.
Some peculiarities of large time behaviour are also demonstrated. The results obtained for the equation in the general infinite--dimensional form are applied to linear stochastic PDE of second order.

Stochastic differential equations in Hilbert spaces
with multiplicative white noise have been studied in numerous
papers, e.g. G.~Da Prato, M.~Iannelli, L.~Tubaro~\cite{daprato1},
\cite{daprato2}, F.~Flandoli~\cite{flandoli2}, and in Chapter~6
of the monograph by G.~Da Prato and J.~Zabczyk~\cite{daprato}.
The solution to such equations may be viewed as a generalization
of the geometric Brownian motion, which has a wide range of
applications. In all these cases the driving process is the
Brownian motion. Later, S. Bonaccorsi in~\cite{bonaccorsi}
studied mild solutions of equations with additional nonlinear
terms in the drift and diffusion parts, defined by means of the
stochastic evolution system induced by the bilinear equation.

Analogous results have been obtained for bilinear evolution
equations of the same type with fractional Gaussian noise  by
T.~Duncan, B.~Maslowski, B.~Pasik--Duncan~\cite{duncan} (for
\mbox{$H>1/2$}, where $H$ denotes the Hurst parameter of the driving
fractional Brownian motion) and J.~\v Snu\-p\'{a}r\-kov\'{a}
\cite{snuparkova} (for $H<1/2$). In these papers the stochastic integral is understood in the Skorokhod sense, i.e. as the adjoint operator to the Malliavin derivative. On the other hand, semilinear evolution equations with bilinear Stratonovich noise have been studied in \cite{GMS}. It was shown that the equation defines a random dynamical system (which is not true in the case of Skorokhod integration) and the long--time behaviour was dealt with.

The present paper is organized as follows. In Section~\ref{kap.1} the
notion of Skorokhod integral with respect to fractional Brownian motion and
its basic properties are recalled. Section~\ref{kap.2} is devoted
to an extension of a result from~\cite{duncan} on the existence
of a weak solution to the bilinear equation. In
Section~\ref{kap.3} the existence and uniqueness of the mild
solution is proved (Theorem~\ref{mild.res}). If the perturbation $F$ does not
depend on the solution process, the mild solution  of the
corresponding affine equation is the weak one
(Theorem~\ref{main.nelin}). Unlike in the case of the standard Brownian motion, this is no longer true if the equation is semilinear (the perturbation $F$ depends on the solution) as shown by a simple counterexample. Note that
it is not completely clear how to define the candidate on the
``random evolution system'', as explained in Example~\ref{priklad}. Following the ideas
from~\cite{bonaccorsi} this system is used to define the mild
solution of an equation with additional nonlinearity in the drift
part for $H>1/2$. While the mild formulation implies a weak one
in the Wiener case it need not  be true in the case of fractional
Brownian motion unless the perturbation is independent of the
solution process. It is shown that this ``mild'' solution
satisfies certain different equation in the weak sense.

Some large--time behaviour results are also proved. Sections \ref{jzn_mild_res} and \ref{jzn_slab_res} are devoted to the proof of uniqueness of solutions to the affine equation. This problem is nontrivial because the Gronwall lemma is not applicable as in the case of standard Brownian motion. Instead, we prove uniqueness of the mild solution to the bilinear equation inductively, showing uniqueness of the coefficients in the Wiener chaos expansions, and using this result we prove the uniqueness of weak solutions to the affine equation.

\section{Preliminaries}\label{kap.1}
Let $(\Omega,\mathcal{F},\P)$ be a complete probability
space. A stochastic process $B^H=\{B^H_t,t\in [0,T]\}$ is a {\bf fractional Brownian motion} with Hurst parameter $H\in
(0,1)$ if it is a real--valued centered Gaussian process with the
covariance function given by
$$\E[B^H_t B^H_s]=\frac{1}{2}\big(s^{2H}+t^{2H}-|s-t|^{2H}\big),\ s,t\geq 0.$$
In what follows Hurst parameter $H>1/2$ is assumed.

Define the linear operator $\mathcal{K}_H^*:{\mathcal
E}\rightarrow L^2([0,T])$ as
\begin{align*}
\big(\mathcal{K}_H^* f\big)(t)&=C_H\Gamma\Big(H-\frac{1}{2}\Big)
t^{\frac{1}{2}-H}
\Big(I^{H-\frac{1}{2}}_{T-}f_{H-\frac{1}{2}}\Big)(t),
\end{align*}
where $f_{H-\frac{1}{2}}(t)=t^{H-\frac{1}{2}}f(t),\ t\in[0,T],$
$f\in\mathcal E$ is a step function of the form
$$\varphi=\sum_{k=0}^{N-1} a_k\,I_{(t_k,t_{k+1}]},$$
for some $N\in \N$, $0=t_0<t_1<\ldots<t_{N}=T$, $a_k\in \R$,
$k=0,\ldots,N$, $I^{H-\frac{1}{2}}_{T-}$ is a Rie\-mann--Liou\-ville
fractional right--sided integral defined as
$$\big(I^{H-\frac{1}{2}}_{T-}f\big)(t)=\frac{1}{\Gamma(H-\frac{1}{2})}\int_t^T \frac{f(s)}{(s-t)^{\frac{3}{2}-H}}
\d s\quad {\rm for}\  {\rm a.e.}\ t\in[0,T],$$ and
$$C_H= \sqrt{\frac{H(2H-1)}{{\rm
B}\left(2-2H,\,H-\frac{1}{2}\right)}}.$$

Using the operator $\mathcal{K}_H^*$ define the scalar product on
$\mathcal{E}$ as
$$\langle\varphi,\psi\rangle_{\mathcal{H}}:=\big\langle\mathcal{K}^*_H(\varphi),\mathcal{K}^*_H(\psi)\big\rangle_{L^2([0,T])},\ \varphi,\psi\in\mathcal E.$$
Denote  by $(\mathcal{H},\langle\,.\,,\,.\,\rangle_{\mathcal{H}})$
the Hilbert space obtained as the completion of $\mathcal{E}$
with respect to the scalar product
$\langle\,.\,,\,.\,\rangle_{\mathcal{H}}$ and let
$\|\,.\,\|_{\mathcal{H}}$ be the norm induced by
$\langle\,.\,,\,.\,\rangle_{\mathcal{H}}$.

For $\varphi\in\mathcal{E}$ define the stochastic integral with
respect to the fractional Brownian motion
$$I(\varphi)\equiv\int_0^T \varphi(s)\,dB^H_s:=\sum_{k=0}^{N-1} a_k\big(B^H(t_{k+1})-B^H(t_k)\big).$$
Since
$$\E\left[\int_0^T \varphi(s)\d B^H_s\int_0^T \psi(s)\d
B^H_s\right]=\langle\varphi,\psi\rangle_{\mathcal{H}},\
\varphi,\psi\in\mathcal E,
$$
(see \cite{alos}), the integral can be uniquely extended to
$\mathcal{H}$ (the standard notation
$I(\varphi)=B^H(\varphi)=\int_0^T \varphi(r)\d B^H_r$ is also
used) and the operator $\mathcal{K}^*_H$ provides an isometry
between spaces $(\mathcal{H},\|\,.\,\|_{\mathcal{H}})$ and
$L^2\big(\Omega\big)$.

Let $\mathcal{S}$ be a set of smooth cylindrical random variables
of the form
\begin{equation}\label{cyl.r.v.}
F=f\big(B^H(\varphi_1),\ldots,B^H(\varphi_n)\big),
\end{equation}
where $n\geq 1$, $f\in\mathcal{C}^{\infty}_b(\R^n)$ ($f$ and all
its partial derivatives are bounded) and
$\varphi_i\in\mathcal{H},$ $i=1,\ldots,n$. The {\bf derivative
operator (Malliavin derivative)} of a smooth cylindrical random
variable $F$ of the form~\Ref{cyl.r.v.} is an
$\mathcal{H}$--valued random variable
$$D^H F= \sum_{i=1}^n \frac{\partial f}{\partial x_i} \big(B^H(\varphi_1),\ldots,B^H(\varphi_n)\big)\varphi_i.$$
The derivative operator $D^H$ is closable from $L^p(\Omega)$ into
$L^p(\Omega;\mathcal{H})$ for any $p\in[1,+\infty)$. Let
$\mathbb{D}^{1,p}_H$ be the Sobolev space obtained as a closure
of $\mathcal{S}$ with respect to the norm
$$\|F\|_{1,p}:=\left(\E\big[|F|^p\big]+\E\big[\|D^H F\|^p_{\mathcal{H}}\big]\right)^{1/p}$$
for any $p\in[1,+\infty)$. Similarly, given a Hilbert space
$\tilde V\subset\mathcal{H}$, set $\mathbb{D}^{1,p}_H(\tilde V)$
for the corresponding Sobolev space of $\tilde V$--valued random
variables.
\begin{Definice}
The \textbf{divergence operator (Skorokhod integral)}
$\delta_H:{\rm Dom}\,\delta_H\rightarrow L^2(\Omega)$ is defined
as the adjoint operator of the derivative ope\-rator
$D^H:L^2(\Omega) \rightarrow L^2(\Omega;\mathcal{H})$, i.e. for
any $u\in {\rm Dom}\,\delta_H$ the duality relationship
$$\E\big[F\delta_H(u)\big]=\E\big[\langle D^H F,u\rangle_{\mathcal{H}}\big]$$
holds for any $F\in\mathbb{D}^{1,2}_H$.

A random variable $u\in L^2(\Omega;\mathcal{H})$ belongs to the
domain ${\rm Dom}\,\delta_H$ if there exists a constant
$c_u<+\infty$ depending only on $u$ such that
$$\big|\E\big[\langle D^H F,u\rangle_{\mathcal H}\big]\big|\leq c_u \|F\|_{L^2(\Omega)}$$
for any $F\in\mathcal{S}$.
\end{Definice}

The useful facts listed below can be found e.g.
in~\cite{nualart1}. Let $|\mathcal{H}|\subset \mathcal{H}$ be a
linear space of measurable functions $\varphi$ on $[0,T]$ such
that
$$\|\varphi\|_{|\mathcal{H}|}^2=\alpha_H \int_0^T\int_0^T |\varphi(r)||\varphi(s)||r-s|^{2H-2}\d r\d s<+\infty,$$
where $\alpha_H=H(2H-1)$. Then $\mathcal{E}$ is dense in
$|\mathcal{H}|$ and $(|\mathcal{H}|,\|\,.\,\|_{|\mathcal{H}|})$
is a~Banach space. Moreover,
$$L^2([0,T])\subset L^{1/H}([0,T])\subset |\mathcal{H}|\subset \mathcal{H},$$
thus there exists a constant $K_e<+\infty$ such that
\begin{equation}\label{vnor_K_H}
\|\mathcal{K}_H^*(\varphi)\|_{L^2([0,T])}=\|\varphi\|_{\mathcal{H}}\leq
K_e \|\varphi\|_{L^2([0,T])}
\end{equation}
for any $\varphi\in\mathcal{H}$. Note that
\begin{equation}\label{int_inkluse}
\mathbb{D}^{1,2}_H(|\mathcal{H}|)\subset
\mathbb{D}^{1,2}_H(\mathcal{H})\subset {\rm Dom}\,\delta_H
\end{equation}
and for some constant $\tilde C_{H,2}<+\infty$
$$\E\big[\delta_H^2(u)\big]\leq \tilde C_{H,2} \left(\E\big[\|u\|^2_{|\mathcal{H}|}\big] +
\E\big[\|D^H
u\|^2_{|\mathcal{H}|\otimes|\mathcal{H}|}\big]\right),\
u\in\mathbb{D}^{1,2}_H(|\mathcal{H}|),$$ 
holds, where
$\mathbb{D}^{1,p}_H(|\mathcal{H}|)$ $\big(p\in(1,+\infty)\big)$
contains processes $u\in \mathbb{D}^{1,p}_H(\mathcal{H})$ such
that $u\in|\mathcal{H}|$, $D^Hu\in
|\mathcal{H}|\otimes|\mathcal{H}|\quad \!\P\textrm{--}\,{\rm a.s.}$ and
$$\E\big[\|u\|^p_{|\mathcal{H}|}\big]+\E\big[\|D^Hu\|^p_{|\mathcal{H}|\otimes|\mathcal{H}|}\big]<+\infty.$$
The normed linear space
$\big(|\mathcal{H}|\otimes|\mathcal{H}|,\|\,.\,\|_{|\mathcal{H}|\otimes|\mathcal{H}|}\big)$
is defined in a similar way as
$(|\mathcal{H}|,\|\,.\,\|_{|\mathcal{H}|})$ (for a precise
definition see e.g.~\cite{nualart1}). Hence, for some constant
$C_{H,2}<+\infty$
\begin{equation}\label{odhad.int}
\E\big[\delta_H^2(u)\big]\leq C_{H,2}
\left(\E\big[\|u\|^2_{L^{1/H}([0,T])}\big] + \E\big[\|D^H
u\|^2_{L^{1/H}([0,T]^2)}\big]\right),\
u\in\mathbb{D}^{1,2}_H(|\mathcal{H}|).
\end{equation}
Since the process $B^H$ has an integral representation (see e.g.
\cite{nualart1})
\begin{equation}
B^H_t=\int_0^t \big({\mathcal K}_H^*I_{(0,t]}\big)(s)\d W_s,\
t\geq 0,\label{reprezentace}
\end{equation}
where $W=\{W_t,t\geq 0\}$ is a Wiener process on
$(\Omega,\mathcal{F},\P)$, similar relations are valid for
derivatives and divergence operators, i.e.
\begin{itemize}
\item[(i)] for any $F\in\mathbb{D}^{1,2}_W$
$$\mathcal{K}^*_H (D^H F) = D^W F,$$
where $D^W$ denotes the derivative operator with respect to $W$
and $\mathbb{D}^{1,2}_W$ the correspon\-ding Sobolev space,
\item[(ii)] ${\rm Dom}\,\delta_W=\mathcal{K}^*_H({\rm Dom}\,\delta_H)$ and
\begin{equation}\label{K_int}
\delta_H(u)=\delta_W(\mathcal{K}^*_H u)
\end{equation}
for any $u\in {\rm Dom}\,\delta_H$, where $\delta_W$ denotes the
divergence operator with respect to $W$.
\end{itemize}
\begin{remark}\label{pozn}
The construction (and the properties) of Malliavin derivative and
Skorokhod integral for Hilbert space--valued random variables are
completely analogous.
\end{remark}

\section{Random evolution system}\label{kap.2}
In this short overview section, a result from \cite{duncan} is
slightly extended to obtain a random two--parameter evolution
system representing the solution to the equation
\begin{equation}\label{ns.rce}
\begin{array}{rcl}
{\rm d} Y_t&=&AY_t\d t+BY_t\d B^H_t,\ t>s,\\
Y_s&=&x,
\end{array}
\end{equation}
in a separable Hilbert space $V$ on a finite interval $[0,T]$
with general initial time $s\in [0,T]$ and deterministic initial
value $x\in V$. The driving process $\{B^H_t,t\geq 0\}$ is a
one--dimensional fractional Brownian motion with Hurst parameter
$H>1/2$ on a complete probability space $(\Omega,
\mathcal{F},\mathbb{P})$ and the stochastic integral is
understood in the Skorokhod sense (see \cite{alos} for more
details).

The linear operators $A$ and $B$ on $V$ satisfy
\begin{itemize}
\item[(A1)] the operator $A$ is closed and densely defined with the domain $D:={\rm Dom}(A)$,
\item[(A2)] the resolvent set contains all $\lambda\in\mathbb{C}$ such that ${\rm Re}(\lambda)\geq \omega$ for some fixed $\omega\in\R$ and for some constant $M>0$ the resolvent $R(\lambda,A)$ satisfies
$$\|R(\lambda,A)\|_{\mathcal{L}(V)}\leq \frac{M}{|\lambda-\omega|+1}$$
for all $\lambda\in \mathbb{C}$, ${\rm Re}(\lambda)\geq \omega$,
where $\mathcal{L}(V)$ stands for a space of all linear bounded
operators on $V$,

\item[(B2)] the operator $B$ is closed, densely defined and generates a strongly conti\-nuous group $\{S_B(u),u\in\R\}$ on $V$.
\end{itemize}

The conditions (A1) and (A2) imply that the operator $A$
generates an analytic semigroup $\{S_A(t),0\leq t\leq T\}$ on
$V$. The condition (B2) ensures the existence of constants
$M_B\geq 1$, $\omega_B\geq 0$ such that the inequality
\begin{equation}\label{SB}
\|S_B(u)\|_{\mathcal{L}(V)}\leq M_B\exp\{\omega_B |u|\}
\end{equation}
holds for each $u\in\R$.

For simplicity assume that $\omega <0$ (cf.(A2)). Note that since
the operator $-A$ is sectorial, the fractional powers
$(-A)^{\alpha}$ for $\alpha\in(0,1]$ are well--defined (see e.g.
\cite{pazy}), so the following condition makes sense. Suppose that
\begin{itemize}
\item[(B3)] $B^2$ is closed and
\begin{equation}\label{fracAB}
{\rm Dom}(B^2)\supset {\rm Dom}\big((-A)^{\alpha}\big)
\end{equation}
for some $\alpha\in (0,1)$.
\end{itemize}

\noindent
Define the operators $\bar A(t): D\rightarrow V$ as
$$\bar A(t)=A-Ht^{2H-1}B^2$$
for any $t\in [0,T]$.

\begin{Lemma}\label{anal.syst}
Under the assumptions {\rm (A1), (A2), (B2)} and {\rm (B3)} the
system $\{\bar A(t),t\in [0,T]\}$ generates a strongly continuous
evolution system $\{U(t,s),0\leq s\leq t\leq T\}$ on $V$.
\end{Lemma}
\begin{proof}
See \cite{duncan}.
\end{proof}

\noindent
The system $\{U(t,s),0\leq s\leq t\leq T\}$ satisfies
\begin{align}
&{\rm Im}(U(t,s))\subset D,\nonumber\\
&\|U(t,s)\|_{\mathcal{L}(V)}\leq C_U,\label{odhadU}\\
&\Big\|\frac{\partial}{\partial t}U(t,s)\Big\|_{\mathcal{L}(V)}=\|{\bar A(t)} U(t,s)\|_{\mathcal{L}(V)}\leq \frac{C_U}{t-s},\nonumber\\
&\|{\bar A(t)}U(t,s)({\bar A(s)}-{\bar\omega
I})^{-1}\|_{\mathcal{L}(V)}\leq C_U\nonumber
\end{align}
for some constants $C_U>0$, $\bar\omega\in\R$ and any $0\leq
s<t\leq T$ (see e.g. \cite{tanabe}, Theorem 5.2.1).

\begin{remark}
Instead of (A2) and (B3) we may assume directly that $\{\bar
A(t),t\in [0,T]\}$ generates a strongly continuous evolution
system $\{U(t,s),0\leq s\leq t\leq T\}$ on $V$. Nevertheless, the
condition (A2) can be useful in applications to stochastic
partial differential equations (as shown in \cite{duncan}).
\end{remark}

Let $A^*$ denote the adjoint operator to the operator $A$. Let
${\rm Dom}(A^*)=D^*$ be the domain of $A^*$ and suppose that
\begin{itemize}
\item[(B1)] $D^*\subset {\rm Dom}((B^*)^2)$.
\end{itemize}

\begin{Definice}
A $\big(\mathcal{B}([s,T])\otimes\mathcal{F}\big)$--measurable
stochastic process $\{Y_t,t\in [s,T]\}$ is said to be a
\textbf{weak solution} to the equation {\rm (\ref{ns.rce})} if
for any $y\in D^*$
$$\langle Y_t,y\rangle _V=\langle x,y\rangle _V+\int_s^t \langle Y_r,A^*y\rangle_V\d r+\int_s^t \langle Y_r,B^*y\rangle _V\d B^H_r\quad \!\P\textrm{--}\,a.s.$$
for all $t\in [s,T]$, where the integrals have to be well--defined.
\end{Definice}

\begin{Veta}\label{evol.syst}
Let
\begin{itemize}
\item[(AB)] the operators $A$ and $\{S_B(u),u\in\R\}$ commute on the domain $D$, i.e.
$$S_B(u)Ay=AS_B(u)y$$
for any $u\in\R$ and $y\in D$.
\end{itemize}
The process $\{U_Y(t,s)x,s\leq t\leq T\}$ defined as
\begin{equation}\label{ns.res}
U_Y(t,s)x=S_B(B^H_t-B^H_s)U(t-s,0)x,\ s\leq t\leq T,
\end{equation}
is a weak solution to
the equation~{\rm \Ref{ns.rce}} for any fixed $x\in V$ and
$s\in[0,T]$ under the assumptions {\rm (A1), (A2)} and {\rm (B1),
(B2), (B3)}.
\end{Veta}

\begin{proof}
The proof is completely analogous to the proof of Theorem~2.3
in~\cite{duncan}.
\end{proof}

\begin{remark}
The system $\{U_Y(t,s),0\leq s\leq t\leq T\}$ is not a random
continuous evolution system because it does not possess the
standard composition property.
\end{remark}

\section{Perturbed equation}\label{kap.3}

In this section the equation with a perturbation in the drift part is studied.

Let $H>1/2$ and $\{U_Y(t,s),0\leq s\leq t\leq T\}$ be the system
of operators defined as
$$U_Y(t,s)x:=S_B(B^H_t-B^H_s)U(t-s,0)x,\ x\in V,$$
where $\{U(t,s), 0\leq s\leq t\leq T\}$ is a strongly continuous
evolution system associated with operators $\{A-Ht^{2H-1}B^2,t\in
[0,T]\}$ and $\{S_B(u),u\in\R\}$ is a strongly continuous group
associated with operator $B$ satisfying conditions from
Theorem~\ref{evol.syst}. Note that in the previous section it has
been shown that for any fixed $s\in [0,T]$ the process
$\{U_Y(t,s)x,s\leq t\leq T\}$ is a~weak solution to the equation
\begin{equation}\label{bilin.rce}
\begin{array}{rcl}
\d Y_t&=&AY_t\d t + BY_t\d B^H_t,\ t>s,\\
Y_s&=&x\in V.
\end{array}
\end{equation}

\begin{Veta}\label{mild.res}
Let $F:[0,T]\times V \rightarrow V$ be a measurable function
satisfying
\begin{itemize}
\item[${\rm (i)}_{\rm F}$] there exists a function $\bar L\in L^1([0,T])$ such that
$$\|F(t,x)-F(t,y)\|_V\leq \bar L(t)\|x-y\|_V,\ x,y\in V,\,t\in [0,T],$$
\item[${\rm (ii)}_{\rm F}$] for some function $\bar K\in L^1([0,T])$
$$\|F(t,0)\|_V\leq \bar K(t),\ t\in [0,T].$$
\end{itemize}
Then the equation
\begin{equation}\label{mild.rce}
y(t)=U_Y(t,0)x+\int_0^t U_Y(t,r)F\big(r,y(r)\big)\d r
\end{equation}
has a unique solution in the space $\mathcal{C}([0,T];V)$ for
a.e. $\omega\in\Omega$ and any initial value $x\in V$.
\end{Veta}

\begin{proof}
Fix $x\in V$ and show that the mapping
$$\big(\mathcal{K}(y)\big)(t)=U_Y(t,0)x+\int_0^t U_Y(t,r)F\big(r,y(r)\big)\d r$$
is continuous from $\mathcal{C}([0,T];V)$ into
$\mathcal{C}([0,T];V)$ and that $\mathcal{K}$ is a contraction
mapping.

\noindent
Take $y\in \mathcal{C}([0,T];V)$ and $t,s\in[0,T]$. Then
\begin{align*}
\big\|&\big(\mathcal{K}(y)\big)(t)-\big(\mathcal{K}(y)\big)(s)\big\|_V\leq\|U_Y(t,0)x-U_Y(s,0)x\|_V\\&+\Big\|\int_0^t
U_Y(t,r)F\big(r,y(r)\big)\d r-\int_0^s
U_Y(s,r)F\big(r,y(r)\big)\d r\Big\|_V=I_1+I_2.
\end{align*}
Note that applying \Ref{SB} and by the continuity of trajectories
of $\{B^H_t,t\in[0,T]\}$
\begin{equation}\label{odhadSB}
\begin{array}{l}
\sup_{t\in[0,T]}\|S_B(B^H_t(\omega))\|_{\mathcal{L}(V)}\leq M_B\exp\{\omega_B \|B^H(\omega)\|_{\mathcal{C}([0,T])}\}\leq C_B(\omega),\\
\sup_{s,t\in[0,T]}\big\|S_B\big(B^H_t(\omega)-B^H_s(\omega)\big)\big\|_{\mathcal{L}(V)}\leq
M_B\exp\{2\omega_B \|B^H(\omega)\|_{\mathcal{C}([0,T])}\leq
C_B(\omega)
\end{array}
\end{equation}
hold for some constant $0<C_B(\omega)<+\infty$ depending on
$\omega\in\Omega$.

\noindent
By the strong continuity of $S_B$ and $U(\,.\,,0)$  on $V$ it
follows
\begin{align*}
I_1&=\|U_Y(t,0)x-U_Y(s,0)x\|_V\\&\leq
\big\|\big(S_B(B^H_t)-S_B(B^H_s)\big)U(t,0)x\big\|_V+
\big\|S_B(B^H_s)\big(U(t,0)-U(s,0)\big)x\big\|_V\\&\leq
\big\|\big(S_B(B^H_t)-S_B(B^H_s)\big)U(t,0)x\big\|_V +
C_B(\omega)\big\|\big(U(t,0)-U(s,0)\big)x\big\|_V\xrightarrow[s \to t]{} 0.
\end{align*}
Now, let $t>s$. Then
\begin{align*}
I_2&=\Big\|\int_0^t U_Y(t,r)F\big(r,y(r)\big)\d r-\int_0^s
U_Y(s,r)F\big(r,y(r)\big)\d r\Big\|_V\\&\leq \Big\|\int_0^s
\big(U_Y(t,r)-U_Y(s,r)\big)F\big(r,y(r)\big)\d
r\Big\|_V+\Big\|\int_s^t U_Y(t,r)F\big(r,y(r)\big)\d
r\Big\|_V=J_1+J_2.
\end{align*}
Using \Ref{odhadSB}, \Ref{odhadU} and \Ref{nejv.lin}
\begin{align*}
J_2&=\Big\|\int_s^t U_Y(t,r)F\big(r,y(r)\big)\d r\Big\|_V\leq
\int_s^t C_U\|S_B(B^H_t-B^H_r)\|_{\mathcal{L}(V)}
\big\|F\big(r,y(r)\big)\big\|_V\d r\\&\leq C_U
C_B(\omega)(1+\|y\|_{\mathcal{C}([0,T];V)})\int_s^t \bar C(r)\d
r\longrightarrow 0
\end{align*}
as $s\rightarrow t-$ or $t\rightarrow s+$.\\
Also
\begin{align*}
J_1&=\Big\|\int_0^s
\big(U_Y(t,r)-U_Y(s,r)\big)F\big(r,y(r)\big)\d r\Big\|_V\\&\leq
\Big\|\int_0^s
\big(S_B(B^H_t-B^H_r)-S_B(B^H_s-B^H_r)\big)U(t-r,0)F\big(r,y(r)\big)\d
r\Big\|_V\\&\qquad+ \Big\|\int_0^s
S_B(B^H_s-B^H_r)\big(U(t-r,0)-U(s-r,0)\big)F\big(r,y(r)\big)\d
r\Big\|_V=K_1+K_2.
\end{align*}
Since for any fixed $0\leq r\leq s$
$$\big\|\big(U(t-r,0)-U(s-r,0)\big)F\big(r,y(r)\big)\big\|_V \longrightarrow 0$$
as $s\rightarrow t-$ or $t\rightarrow s+$ and by \Ref{odhadU}
\begin{align*}
\big\|\big(U(t-r,0)-U(s-r,0)\big)F\big(r,y(r)\big)\big\|_V&\leq
2C_U\big\|F\big(r,y(r)\big)\big\|_V\\&\leq 2C_U
(1+\|y\|_{\mathcal{C}([0,T];V)})\bar C(r)\in L^1([0,T]),
\end{align*}
the convergence
\begin{align*}
K_2&=\Big\|\int_0^s
S_B(B^H_s-B^H_r)\big(U(t-r,0)-U(s-r,0)\big)F\big(r,y(r)\big)\d
r\Big\|_V\\&\leq C_B(\omega) \int_0^s
\big\|\big(U(t-r,0)-U(s-r,0)\big)F\big(r,y(r)\big)\big\|_V\d r
\longrightarrow 0
\end{align*}
is obtained as $s\rightarrow t-$ or $t\rightarrow s+$ by the
Lebesgue dominated convergence theorem. Note that the set
\begin{align*}
K:=\Big\{\bar y\in V;\,\exists\,0\leq s_1\leq t_1\leq T\quad \bar
y=\int_0^{s_1} S_B(-B^H_r)U(t_1-r,0)F\big(r,y(r)\big)\d r\Big\}
\end{align*}
is compact (being a continuous image of a compact set) and
$$\lim_{t\rightarrow s}\ \sup_{z\in K} \big\|\big(S_B(B^H_t)-S_B(B^H_s)\big)z\big\|_V = 0.$$
Therefore
\begin{align*}
K_1&= \Big\|\int_0^s
\big(S_B(B^H_t-B^H_r)-S_B(B^H_s-B^H_r)\big)U(t-r,0)F\big(r,y(r)\big)\d
r\Big\|_V\\&= \Big\|\big(S_B(B^H_t)-S_B(B^H_s)\big) \int_0^s
S_B(-B^H_r)U(t-r,0)F\big(r,y(r)\big)\d r\Big\|_V\\&\leq
\sup_{z\in K} \big\|\big(S_B(B^H_t)-S_B(B^H_s)\big)z\big\|_V
\longrightarrow 0
\end{align*}
as $s\rightarrow t-$ or $t\rightarrow s+$. Thus
$$\big\|\big(\mathcal{K}(y)\big)(t)-\big(\mathcal{K}(y)\big)(s)\big\|_V\longrightarrow 0$$
as $s\rightarrow t-$ or $t\rightarrow s+$ and the function
$t\mapsto\big(\mathcal{K}(y)\big)(t)$ is continuous on the
interval $[0,T]$ for any $y\in\mathcal{C}([0,T];V)$.

\noindent
For any $y_1,y_2\in\mathcal{C}([0,T];V)$, $t\in[0,T]$ and $T>0$
small enough there exists a constant $0<L_T(\omega)<1$ such that
\begin{align*}
\big\|&\big(\mathcal{K}(y_1)\big)(t)-\big(\mathcal{K}(y_2)\big)(t)\big\|_V=\Big\|\int_0^t
U_Y(t,r)\big(F\big(r,y_1(r)\big)-F\big(r,y_2(r)\big)\big)\d
r\Big\|_V\\& \leq C_B(\omega) C_U \int_0^t
\big\|\big(F\big(r,y_1(r)\big)-F\big(r,y_2(r)\big)\big)\big\|_V\d
r\\&\leq C_B(\omega) C_U \|y_1-y_2\|_{\mathcal{C}([0,T];V)}
\int_0^T\bar L(r)\d r\leq L_T(\omega)
\|y_1-y_2\|_{\mathcal{C}([0,T];V)}
\end{align*}
holds so that $\mathcal{K}$ is a contraction mapping. Hence, by
the Banach fixed--point theo\-rem there exists a unique solution
to the equation \Ref{mild.rce} for $T$ small enough. Applying
standard methods a unique continuous solution to \Ref{mild.rce}
for any $T>0$ can be obtained.
\end{proof}

Consider an equation with a nonlinear perturbation of a drift part
\begin{equation}\label{nelin.rceD}
\begin{array}{rcl}
\d X_t&=&AX_t\d t +F(t,X_t)\d t + BX_t\d B^H_t,\\
X_0&=&x\in V.
\end{array}
\end{equation}

\begin{Definice}
A $\big(\mathcal{B}([0,T])\otimes\mathcal{F}\big)$--measurable
process $\{X_t,t\in[0,T]\}$ is a \textbf{weak
solution} to the equation {\rm (\ref{nelin.rceD})} if for any
$y\in D^*$
\begin{align*}
\langle X_t,y\rangle _V&=\langle x,y\rangle _V+\int_0^t \langle
X_r,A^*y\rangle_V\d r+\int_0^t \langle F(r,X_r),y\rangle_V\d
r+\int_0^t \langle X_r,B^*y\rangle _V\d B^H_r\quad \P\!\!\textrm{
--}\,a.s.
\end{align*}
for all $t\in [0,T]$, where the integrals have to be well--defined.
\end{Definice}

\begin{remark}
\begin{itemize}
\item[(i)] The conditions ${\rm (i)}_{\rm F}$ and ${\rm (ii)}_{\rm F}$ imply
\begin{equation}\label{nejv.lin}
\|F(t,x)\|_V\leq \bar C(t)(1+\|x\|_V),\ x\in V,\,t\in[0,T].
\end{equation}
for a function $\bar C\in L^1([0,T])$.
\item[(ii)] In the Wiener case $H=1/2$ the solution to the equation~\Ref{mild.rce} is the so--called mild solution to the equation
$$
\begin{array}{rcl}
\d X_t&=&AX_t\d t +F(t,X_t)\d t + BX_t\d W_t,\\
X_0&=&x\in V.
\end{array}
$$
In this case, S.~Bonaccorsi (\cite{bonaccorsi}) has shown that
the solution to the equation~\Ref{mild.rce} is
also the weak one. This in general is not true for the equation (\ref{nelin.rceD}) as is
shown in the simple counterexample below.
\end{itemize}
\end{remark}

\begin{Priklad}\label{priklad}
Consider a one--dimensional equation
\begin{equation}\label{1-dim.rce}
\d X_t = a X_t\d t + b X_t\d B^H_t,\ X_0 = 1,
\end{equation}
where $a,b\in\R$ are nonzero constants. Note that
the equation~\Ref{1-dim.rce} takes the form (\ref{nelin.rceD}) with
$F(t,x)=ax$, $A=0$, $B=b I$ and $x_0=1$.

\noindent
The solution to the equation~\Ref{1-dim.rce} is given by the formula
$$X_t=\exp\left\{bB^H_t-\frac{1}{2}b^2t^{2H}+at\right\},\ t\in[0,T],$$
and the random evolution system corresponding to the above choice of coefficients is

$$U_Y(t,s)=
S_B(B^H_t-B^H_s)U(t,s)=\exp\left\{b\big(B^H_t-B^H_s\big)-\frac{1}{2}b^2\big(t-s)^{2H}\right\},\
0\leq s\leq t\leq T.$$
\noindent
It is now easy to compute that the solution $\{X_t,t\in[0,T]\}$ DOES NOT satisfy the mild formula

\begin{equation}\label{1-dim.mild}
y(t)= U_Y(t,0) + \int_0^t  U_Y(t,r)F(r,y(r))\d r.
\end{equation}
\noindent
Note that if we define the system $\{\bar U_Y(t,s),0\leq s\leq t\leq T\}$ as
\begin{align*}
\bar U_Y(t,s)&=
S_B(B^H_t-B^H_s)U(t,s)=\exp\left\{b(B^H_t-B^H_s)-\frac{1}{2}b^2\big(t^{2H}-s^{2H}\big)\right\},\
0\leq s\leq t\leq T,
\end{align*}
the above mild formula holds if $U_Y$ is replaced by $\bar U_Y$.
\end{Priklad}

\begin{remark}
Let the assumptions of Theorem~\ref{evol.syst} be satisfied. Then
the system $\{\bar U_Y(t,s),0\leq s\leq t\leq T\}$ defined as
\begin{equation}\label{a_evol.syst}
\bar U_Y(t,s)x= S_B(B^H_t-B^H_s)U(t,s)x,\ x\in V,\,0\leq s\leq
t\leq T,
\end{equation}
is a weak solution to the equation
$$
\begin{array}{rcl}
\d Y_t &=& A(t)Y_t\d t + H\big((t-s)^{2H-1}-t^{2H-1}\big)B^2Y_t\d t + BY_t\d B^H_t,\ t>s,\\
Y_s&=&x.
\end{array}
$$
This result can be obtained in the same way as
Theorem~\ref{evol.syst}. The system $\{\bar U_Y(t,s),0\leq s\leq
t\leq T\}$ defined in Example~\ref{priklad} is a particular case
of~\Ref{a_evol.syst}. Moreover, this system has a~com\-position
property.

It is easy to verify the fact that $\{U_Y(t,s),0\leq s \leq
t\leq T\}$ does not possess the composition property, which means that
the equation~\Ref{ns.rce} does not define a cocycle in the usual
way. On the other hand, in~\cite{bartek} it has been proved (for
the case of stochastic equation with homogeneous right hand side
and bilinear fractional noise) that the cocycle property does
hold in the case when stochastic integration in Stratonovich
sense is considered.
\end{remark}

The natural question is whether there is a chance to obtain a weak solution as the unique solution to the equation~\Ref{mild.rce}. The positive answer gives the next theorem but only under the restrictive assumption on $F$ that it does not depend on the space variable.

\begin{Veta}\label{main.nelin}
Assume that the measurable function $F:[0,T]\rightarrow V$ is affine and that \mbox{$\|F\|_V\in L^2([0,T])$.} Then the unique continuous solution $\{X_t,t\in [0,T]\}$ to the equation
\begin{equation}\label{am_rce}
X^M_t=U_Y(t,0)x+\int_0^t U_Y(t,r)F(r)\d r
\end{equation}
stated in Theorem~\ref{mild.res} is a weak solution to the equation
\begin{equation}\label{a_rce}
\begin{array}{rcl}
\d X_t&=&AX_t\d t +F(t)\d t + BX_t\d B^H_t,\\
X_0&=&x\in V.
\end{array}
\end{equation}
\end{Veta}

\noindent
The main idea of the proof is to use standard and stochastic Fubini theorem for the Skorokhod integral stated in \cite{leon}, Lemma~2.10, or \cite{nualart}, Exercise~3.2.8.

\begin{Lemma}\label{fubini}
Consider a random field $\{u(t,x),t\in [0,T],x\in G\}$, where $G\subset\R$ is a bounded set, such that
\begin{itemize}
\item[${\rm (i)}_{\rm W}$] $u\in L^2(\Omega\times [0,T]\times G)$,
\item[${\rm (ii)}_{\rm W}$] $u(\,.\,,x)\in {\rm Dom}\,\delta_W$ for a.e. $x\in G$,
\item[${\rm (iii)}_{\rm W}$] $\E\left[\int_G \left(\int_0^T u(t,x)\d W_t\right)^2\d x\right]<+\infty.$
\end{itemize}
Then the process $\left\{\int_G u(t,x)\d x,t\in[0,T]\right\}\in {\rm Dom}\,\delta_W$ and
$$\int_0^T\left(\int_G u(t,x)\d x\right)\d W_t = \int_G \left(\int_0^T u(t,x)\d W_t\right)\d x.$$
\end{Lemma}

\noindent
Due to the relationship between Skorokhod integral with respect to Wiener process and fractional Brownian motion (see \Ref{K_int} or \cite{nualart1} for more detailes) ${\rm (ii)}_{\rm W}$, ${\rm (iii)}_{\rm W}$ are equi\-va\-lent to
\begin{itemize}
\item[${\rm (ii)}_{\rm H}$]  $u_H(\,.\,,x)\in {\rm Dom}\,\delta_H$ for a.e. $x\in G$,
\item[${\rm (iii)}_{\rm H}$] $\E\left[\int_G \left(\int_0^T u_H(t,x)\d B^H_t\right)^2\d x\right]<+\infty,$
\end{itemize}
respectively, where $u_H(t,x)=\big(\mathcal{K}^*_H\big)^{-1}\big(u(\,.\,,x)\big)(t),t\in[0,T]$. The conclusion of Lemma~\ref{fubini} can be reformulated in the following way. The process $\left\{\int_G u_H(t,x)\d x,t\in[0,T]\right\}\in {\rm Dom}\,\delta_H$ and
$$\int_0^T\left(\int_G u_H(t,x)\d x\right)\d B^H_t = \int_G \left(\int_0^T u_H(t,x)\d B^H_t\right)\d x.$$

\noindent
The proof of Theorem~\Ref{main.nelin} is based on the following lemma.

\begin{Lemma}\label{fub1}
The equalities
\begin{equation}\label{zam1}
\int_0^t\int_0^r \big\langle U_Y(r,v)F(v),A^*\zeta\big\rangle_V\d v\d r= \int_0^t\int_v^t \big\langle U_Y(r,v)F(v),A^*\zeta\big\rangle_V\d r\d v
\end{equation}
and
\begin{align*}
\int_0^t \!\int_0^r \!\big\langle U_Y(r,v)F(v),B^*\zeta\big\rangle _V\d v\d B^H_r=\int_0^t \!\int_v^t \!\big\langle U_Y(r,v)F(v),B^*\zeta\big\rangle _V\d B^H_r\!\d v
\end{align*}
hold $\quad
\P\!\textrm{--}\, a.s.$ for any $t\in[0,T]$ and fixed $\zeta\in D^*$.
\end{Lemma}

\begin{proof}
It is necessary to verify the assumptions of standard and stochastic Fubini theorem.\\
Notice that the Fernique theorem (see e.g. \cite{fernique}) yields that there exists a~random variable $C_{B^H}(\omega)$ such that $C_{B^H}\in  L^q(\Omega)$ for any $q\in[1,+\infty)$ and
\begin{equation}\label{odhad_fernique}
M_B\exp\{l\omega_B \|B^H(\omega)\|_{\mathcal{C}([0,T])}\}\leq C_{B^H}(\omega),\ \omega\in\Omega,l=1,2.
\end{equation}
Since by \Ref{odhad_fernique} and \Ref{odhadU}
\begin{align*}
\int_0^t\int_0^r \big|\big\langle U_Y(r,v)F(v),A^*\zeta\big\rangle_V\big|\d v\d r&\leq \int_0^T\int_0^T C_{B^H}(\omega)C_U\|F(v)\|_V\|A^*\zeta\|_V\d v\d r\\&\leq K(\omega)\int_0^T \|F(v)\|_V\d v<+\infty
\end{align*}
for a.e. $\omega\in\Omega$, \Ref{zam1} follows by the standard Fubini theorem.\\
Denote
\begin{align*}
u_H(r,s)&=\big\langle U_Y(r,s)F(s),B^*\zeta\big\rangle _V,\ 0\leq s\leq r\leq t,\\
u(r,s)&=\big(\mathcal{K}^*_H u_H(\,.\,,s)\big)(r),\ 0\leq s\leq r\leq t,
\end{align*}
and verify that ${\rm (i)}_{\rm W},{\rm (ii)}_{\rm H}$ and ${\rm (iii)}_{\rm H}$ hold for the corresponding processes. First show that $u\in L^2([0,t]^2\times\Omega)$. Using \Ref{vnor_K_H}
\begin{align*}
\E\left[\int_0^t\int_0^t u^2(r,s)\d r\d s\right]&\leq K_e \E\left[\int_0^t\int_0^t u_H^2(r,s)\d r\d s\right]\\&\leq K_e\E\left[\int_0^t\int_0^t \big(C_{B^H}(\omega)C_U\|F(s)\|_V\|B^*\zeta\|_V\big)^2\d r\d s\right]<+\infty,
\end{align*}
and ${\rm (i)}_{\rm W}$ follows. To prove ${\rm (ii)}_{\rm H}$ it sufficies to show (in the view of \Ref{int_inkluse}) that $u_H(\,.\,,s)\in \mathbb{D}^{1,2}_H(|\mathcal{H}|)$ for a.e. $s\in [0,t]$ which is true whenever
\begin{equation}\label{odhad_uH}
\max\left\{\sup_{r\in [0,t]}\E\big[u_H^2(r,s)\big], \sup_{r\in [0,t]}\sup_{v\in [0,t]}\E\big[(D_v^H u_H(r,s))^2\big]\right\}<+\infty
\end{equation}
for a.e. $s\in [0,t]$. Since
$$D_v^H u_H(r,s)=\big\langle U_Y(r,s)F(s),(B^*)^2\zeta\big\rangle _V I_{(s,r]}(v)$$
the inequalities
\begin{align*}
\sup_{r\in [0,t]}\sup_{v\in [0,t]}&\E\big[(D_v^H u_H(r,s))^2\big]\\&\leq \sup_{r\in [0,t]} \E\big[\big(C_{B^H}(\omega)C_U\|F(s)\|_V\|(B^*)^2\zeta\|_V\big)^2\big] = K\|F(s)\|^2_V<+\infty
\end{align*}
and
$$\sup_{r\in [0,t]}\E\big[u_H^2(r,s)\big]\leq \E\big[(C_{B^H}(\omega)C_U\|F(s)\|_V\|B^*\zeta\|_V)^2\big]\leq K\|F(s)\|^2_V<+\infty$$
hold for a.e. $s\in [0,t]$ which completes the proof of~\Ref{odhad_uH}.\\
Finally, applying the estimate on the Skorokhod integral \Ref{odhad.int} and the previous part of the proof of~\Ref{odhad_uH}
\begin{align*}
\E&\left[\int_0^t\Big(\int_0^t u_H(r,s)\d B^H_r\Big)^2\d s\right]=\int_0^t\E\left[\Big(\int_0^t u_H(r,s)\d B^H_r\Big)^2\right]\d s\\&\quad\leq C_{H,2} \int_0^t \Big(\E\big[\|u_H(\,.\,,s)\|_{L^2([0,t])}^2\big]+ \E\big[\|D^H u_H(\,.\,,s)\|_{L^2([0,t]^2)}^2\big]\Big)\d s\\&\quad\leq C_{H,2} \int_0^t (t+t^2) K\|F(s)\|^2_V\d s<+\infty
\end{align*}
holds and ${\rm (iii)}_{\rm H}$ follows.
\end{proof}

\begin{proof}[Proof of Theorem~\ref{main.nelin}] $\quad$ Fix $\zeta\in D^*$. Since $\{X_t,t\in [0,T]\}$ satisfies~\Ref{am_rce} and $\{U_Y(t,s)x,s\leq t\leq T\}$ is a~weak solution to the equation \Ref{bilin.rce}
\begin{align*}
\int_0^t& \langle X_r,A^*\zeta\rangle_V\d r+\int_0^t \langle X_r,B^*\zeta\rangle _V\d B^H_r= \int_0^t \big\langle U_Y(r,0)x,A^*\zeta\big\rangle_V\d r\\&\qquad+\int_0^t\int_0^r \big\langle U_Y(r,v)F(v),A^*\zeta\big\rangle_V\d v\d r+ \int_0^t \big\langle U_Y(r,0)x,B^*\zeta\big\rangle _V\d B^H_r\\&\qquad+ \int_0^t \int_0^r \big\langle U_Y(r,v)F(v),B^*\zeta\big\rangle _V\d v\d B^H_r\\&= \langle U_Y(t,0)x,\zeta \rangle_V-\langle x,\zeta \rangle_V+ \int_0^t\int_v^t \big\langle U_Y(r,v)F(v),A^*\zeta\big\rangle_V\d r\d v \\&\qquad+\int_0^t \int_v^t \big\langle U_Y(r,v)F(v),B^*\zeta\big\rangle _V\d B^H_r\d v\quad \P\!\textrm{--}\,{\rm a.s.}
\end{align*}
holds for any $t\in[0,T]$, where in the last equality Lemma~\ref{fub1} is used.\\
Applying again that $\{U_Y(t,s)x,s\leq t\leq T\}$ is a~weak solution to the equation \Ref{bilin.rce}
\begin{align*}
\int_0^t& \langle X_r,A^*\zeta\rangle_V\d r+\int_0^t \langle X_r,B^*\zeta\rangle _V\d B^H_r\\&= \langle U_Y(t,0)x,\zeta \rangle_V-\langle x,\zeta \rangle_V+ \int_0^t \big\langle U_Y(t,v)F(v),\zeta\big\rangle_V\d v-\int_0^t \langle F(v),\zeta\rangle _V\d v\\&= \langle X_t,\zeta \rangle_V-\langle x,\zeta \rangle_V-\int_0^t \langle F(v),\zeta\rangle _V\d v\quad \P\!\textrm{--}\,{\rm a.s.}
\end{align*}
is obtained for any $t\in [0,T]$ and the conclusion follows.
\end{proof}

\begin{remark}
In  view of Example~\ref{priklad}, one can ask whether the solution to the
equation~\Ref{mild.rce} is a weak one to some equation.
A partial answer is given by the subsequent Theorem the proof of which is similar to the one of Theorem \ref{main.nelin} (but more technical) and is omitted.

\begin{Veta}\label{main1.nelin}
Let the assumptions of Theorem~\ref{mild.res} hold and
$\{X_t,t\in[0,T]\}$ be the solution to the equation~{\rm
\Ref{mild.rce}} such that there exists a constant $C_X<+\infty$
\begin{equation}\label{omez_DF}
\max\left\{\sup_{t\in [0,T]}\E\|X_t\|^4_V, \sup_{t\in
[0,T]}\sup_{v\in [0,T]}\E\|D_v^H X_t\|^4_V\right\}\leq C_X.
\end{equation}
In addition, let $F$ be Fr\'{e}chet differentiable with respect to
the space variable for any time $t\in [0,T]$. Suppose that there
exists a function $C\in L^4([0,T])$ such that
\begin{equation}\label{omez_F}
\max\{\|F(t,x)\|_V,\|F'_x(t,x)\|\}\leq C(t),\ t\in [0,T],
\end{equation}
holds. Then $\{X_t,t\in[0,T]\}$ is a solution to the integral
equation
\begin{align*}
X_t&=x+ \int_0^t AX_r\d r +\int_0^t F(r,X_r)\d r + \int_0^t BX_r\d B^H_r\\
&\qquad+\int_0^t\alpha_H\int_0^T\int_r^t
|v-w|^{2H-2}BU_Y(v,r)F'_x(r,X_r)D^H_w X_r\d v\d w\d r
\end{align*}
in a weak sense, i.e. for any $y\in D^*$
\begin{align*}
\langle X_t,y\rangle_V &=\langle x,y\rangle_V + \int_0^t \langle
X_r,A^*y\rangle_V\d r + \int_0^t \langle F(r,X_r),y \rangle_V \d
r +\int_0^t \langle X_r,B^*y\rangle _V\d B^H_r \\&\qquad+\int_0^t
\alpha_H\int_0^T\int_r^t|v-w|^{2H-2}\big\langle
U_Y(v,r)F'_x(r,X_r)D^H_w X_r,B^*y\big\rangle_V\d v\d w\d r \quad
\P\!\textrm{--}\, a.s.
\end{align*}
for all $t\in [0,T]$.
\end{Veta}

\begin{remark}
The condition \Ref{omez_DF} implies that $X\in
\mathbb{D}^{1,4}_H(|\mathcal{H}|)$.
\end{remark}

\end{remark}

\begin{Priklad}
Consider the stochastic parabolic equation of the second order
with the additional affine term in a drift part
\begin{align}\label{PDE}
\frac{\partial u}{\partial t}(t,x)&=\big(Lu(t,\,.\,)\big)(x)+f(t,x)+bu(t,x)\frac{\d B^H}{\d t},\\
u(0,x)&=x_0(x),\ x\in \mathcal{O},\nonumber\\
u(t,x)&=0,\ (t,x)\in [0,T]\times \partial\mathcal{O},\nonumber
\end{align}
where $\mathcal{O}\subset \R^d$ is a bounded domain with the
boundary of class $\mathcal{C}^2$, $b\in\R\setminus \{0\}$ and
$$\big(Lu(t,\,.\,)\big)(x)=a_0(x)u(t,x)+\sum_{i=1}^d a_i(x)\frac{\partial u}{\partial x_i}(t,x)+\sum_{i,j=1}^d a_{ij}(x)\frac{\partial^2 u}{\partial x_i\partial x_j}(t,x)$$
is a strongly elliptic operator on $\mathcal{O}$. 

\noindent
Suppose that the functions
$a_0,a_i,a_{ij}\in\mathcal{C}^{\infty}(\bar{\mathcal{O}})$ for
$i,j=1,\ldots, d$. Let $V=L^2(\mathcal{O})$. Assume that the
mapping $F:[0,T]\rightarrow V;\, F(t):=f(t,\,.\,),$ satisfies
$F\in L^2\big([0,T];V\big)$.

\noindent
Equation~\Ref{PDE} can be rewritten in the form \Ref{a_rce}, where
$$\big(A u(t,\,.\,)\big)(x)=\big(Lu(t,\,.\,)\big)(x),$$
with ${\rm Dom}(A)=D=H^2(\mathcal{O})\cap H^1_0(\mathcal{O})$ and
$B=b I\in\mathcal{L}(V)$. The adjoint operator $A^*$ has the same
form as the operator $A$ (possibly, with different coefficients),
hence ${\rm Dom}(A^*)=D$. In this case the
assumptions of Theorem~\ref{main.nelin} (including those  of
Theorems~\ref{evol.syst} and~\ref{mild.res}) are satisfied,
therefore the process $\{X_t,t\in[0,T]\}$ defined as
$$X_t=U_Y(t,0)x_0+\int_0^t U_Y(t,r)F(r)\d r$$
is a weak solution to the equation~\Ref{PDE}. Note that the
process $\{U_Y(t,s),0\leq s\leq t\leq T\}$ defined in
Theorem~\ref{evol.syst} has the form
$$U_Y(t,s)=\exp\Big\{b(B^H_t-B^H_s)-\frac{1}{2}b^2(t-s)^{2H}\Big\}S_L(t-s),\ 0\leq s\leq t\leq T,$$
where $\{S_L(t),t\in [0,T]\}$ is the strongly continuous
semigroup on $V$ generated by operator $A$.
\end{Priklad}

Theorem~\ref{main.nelin} may serve as a useful tool for analysis
of a behaviour of the weak solutions to~\Ref{a_rce}. As an example
a simple result on large--time behaviour of the solution to the
equation
\begin{equation}\label{chov.rce}
\begin{array}{rcl}
\d X_t&=&\big(AX_t + F(t)\big)\d t + b X_t\d B^H_t,\ t>0,\\
X_0&=&x,
\end{array}
\end{equation}
is provided, where $A:{\rm Dom}(A)\subset V\rightarrow V$ is the
generator of a strongly continuous semigroup $\{S_A(t),t\geq 0\}$
and $b\in\R\setminus\{0\}$.

It is easily seen that
$$U_Y(t,s)=\exp\Big\{b(B^H_t-B^H_s)-\frac{1}{2}b^2(t-s)^{2H}\Big\}S_A(t-s),\ 0\leq s\leq t<+\infty,$$
and since there exist some constants $M>0,\omega\in\R$, such that
$$\|S_A(t)\|_{\mathcal{L}(V)}\leq M {\rm e}^{\omega t},\ t\geq 0,$$
the inequality
\begin{equation}\label{evol.odhad}
\|U_Y(t,s)\|_{\mathcal{L}(V)}\leq M
\exp\Big\{b(B^H_t-B^H_s)-\frac{1}{2}b^2(t-s)^{2H}+\omega(t-s)\Big\},\
0\leq s\leq t<+\infty,
\end{equation}
is obtained.
\begin{Tvrzeni}
Assume that $F\in L^2\big([0,T];V\big)$. Then the solution
$\{X_t,t\geq 0\}$ to the equation~{\rm \Ref{chov.rce}} satisfies
$$\|X_t\|_V\leq y(t),\ t\geq 0,\quad
\P\!\textrm{--}\, a.s.,$$
where $y$ is a solution to one--dimensional equation
\begin{equation}\label{a_mild.rce}
\begin{array}{rcl}
\d y(t)&=&\big(\omega y(t) + \|F(t)\|_V\big)\d t + b y(t)\d B^H_t,\ t>0,\\
y(0)&=&M\|x\|_V.
\end{array}
\end{equation}
\end{Tvrzeni}

\begin{proof}
The proof easily follows from \Ref{am_rce} and \Ref{evol.odhad}
because
\begin{align}
\|X_t\|_V&\leq M \exp\Big\{bB^H_t-\frac{1}{2}b^2t^{2H}+\omega
t\Big\}\|x\|_V \nonumber\\&\qquad+ \int_0^t
\exp\Big\{b\big(B^H_t-B^H_s\big)-\frac{1}{2}b^2(t-s)^{2H}+\omega(t-s)\Big\}M\|F(s)\|_V\d
s \label{a_odhad}
\end{align}
and by Theorem~\ref{main.nelin} the right--hand side of
\Ref{a_odhad} is exactly the formula for the solution
to~\Ref{a_mild.rce}.
\end{proof}

\begin{Dusledek}
For each $p\geq 1$ there exists a constant $c_p>0$ depending only
on $p$ such that
\begin{align}
\E\big[\|X_t\|^p_V\big]&\leq c_p M
\exp\Big\{\frac{(p^2-p)b^2}{2}t^{2H}+p\omega
t\Big\}\|x\|_V^p\nonumber\\&\qquad + M t^{p-1}\int_0^t
\exp\Big\{\frac{(p^2-p)b^2}{2}(t-s)^{2H}+p\omega
(t-s)\Big\}\|F(s)\|_V^p\d s,\ t\geq 0.\label{a1_odhad}
\end{align}
In particular, if $F(t)\equiv F$ does not depend on $t\geq 0$,
for each $\epsilon>0$ there exists $C_{\epsilon}>0$ such that
\begin{equation}\label{a2_odhad}
\E\big[\|X_t\|^p_V\big]\leq C_{\epsilon} \exp\{(\hat c
+\epsilon)t^{2H}\},\ t\geq 0,
\end{equation}
holds with $\hat c = 1/2 b^2(p^2-p)$.
\end{Dusledek}

\begin{proof}
The inequality~\Ref{a1_odhad} easily follows from~\Ref{a_odhad}
if we take into account that
$$\E\left[\exp\left\{p\Big(b\big(B^H_t-B^H_s\big)-\frac{1}{2}b^2(t-s)^{2H}+\omega(t-s)\Big)\right\}\right] = \exp\big\{\hat c (t-s)^{2H}+p\omega (t-s)\big\}$$
for all $0\leq s\leq t$ and apply the H\"{o}lder inequality on
the second term on the right--hand side of~\Ref{a_odhad}. The
inequality~\Ref{a2_odhad} is an immediate consequence
of~\Ref{a1_odhad}.
\end{proof}

\begin{remark}
A simple one--dimensional example shows that the bound $\hat c$
in~\Ref{a2_odhad} is, in some sense, sharp. Take $V=\R,\,
A=\omega,\, F=0,\, x\neq 0$, then
$$|X_t|^p=|x|^p\exp\Big\{p\omega t -\frac{1}{2} b^2p t^{2H} + pbB^H_t\Big\},\ t\geq 0,\,p>1,$$
hence for each $\epsilon >0$ there exists $\tilde C_{\epsilon}>0$
such that
$$\E\big[|X_t|^p_V\big]=|x|^p\exp\big\{\hat c t^{2H} + p\omega t\big\}\geq \tilde C_{\epsilon}\exp\{(\hat c - \epsilon)t^{2H}\},\ t\geq 0.$$
It means that for $p>1$ the $p$--th moment of the solution to
linear equation may be destabilized by adding bilinear fractional
noise of the form $bX_t \dot{B}^H_t,b\neq 0$, even if the
original equation is stable (here $\omega<0$). It may be
interesting to note that from~\cite{daprato2}, Remark 3.7,
applied to the same example it follows that the solution tends to
zero pathwise exponentially fast as $t\rightarrow +\infty$, even
if the equation without noise is not stable (i.e. $\omega>0$).
\end{remark}

\section{Uniqueness of mild solution}\label{jzn_mild_res}
This section is devoted to the proof of the uniqueness of the mild
solution to the equation
\begin{equation}\label{jzn.rce}
\d X_t = A X_t\d t + B X_t\d B^H_t,\ X_0 = x,
\end{equation}
on the interval $[0,T]$. Let $H>1/2$ and recall that $S_B$ is a strongly continuous group generated by $B$ and $U$ is a strongly continuous evolution system associated with operators $A-Ht^{2H-1}B^2,t\in[0,T]$.

\begin{Veta}\label{ex.lin.}
Let the conditions {\rm (A1), (A2), (AB)} be satisfied and let
$B\in\mathcal{L}(V)$. Then
the process $X=\{X_t,t\in[0,T]\}$ given by
\begin{equation}\label{res}
X_t=S_B(B^H_t)U(t,0)x
\end{equation}
is a mild solution to the equation~{\rm\Ref{jzn.rce}}, i.e.
$$X_t=S_A(t)x+\int_0^t S_A(t-r)BX_r\d B^H_r\quad
\P\!\textrm{--}\, a.s.$$
for all $t\in [0,T]$, where $\{S_A(t),t\geq 0\}$ is an analytic semigroup generated by~$A$.
\end{Veta}

\begin{proof}
See~\cite{duncan}.
\end{proof}

The aim is to show, that $X$ defined by~\Ref{res} is a unique mild
solution to~\Ref{jzn.rce}. The idea of the proof is to use the
fractional Wiener chaos decomposition as in the
paper~\cite{perez-abreu}, where the result is proved in a
one--dimensional case.

The construction of multiple fractional integrals and fractional
Wiener chaos decomposition that are used below, are made only for
real--valued random variables. However, all remains true for
Hilbert space--valued random variables (see e.g.~\cite{nourdin}).
For the simplicity the Hilbert space--valued notation is the same
as the real--valued notation.

Let $\mathcal{H}^{\otimes n}$ denote the $n$th tensor product of
$\mathcal{H}$ for any $n\geq 2$. Set $\mathcal{H}^{\otimes
1}\equiv \mathcal{H}$ and $\mathcal{H}^{\otimes 0}\equiv \R$ or
$V$, respectively.

\begin{Definice}
Let $n\in \N$. For $f\in \mathcal H^{\otimes n}$ symmetric the
multiple fractional integral of order $n$ of $f$ is defined as
$$I^{H}_n(f)=\delta_H^n(f),$$
where $\delta_H^n$ is the multiple divergence operator (Skorokhod
integral) of order $n$ (for the definition see e.g.~\cite{nourdin}).
\end{Definice}
\noindent
Note that $\delta_H^1\equiv \delta_H$.

As in the Wiener case the functions $F\in
L^2(\Omega;\mathcal{G},\P)$ (where $\mathcal{G}$ denotes the
$\sigma$--field generated by $\{B^H_t,t\in[0,T]\}$) admit the
unique fractional Wiener chaos decomposition
$$F=\E[F]+\sum_{n=1}^{+\infty} I^H_n(f_n),$$
where $f_n\in \mathcal H^{\otimes n}$ are symmetric elements
which are uniquely determined (see~\cite{nourdin} or
\cite{nualart}). Let
$$\mathcal{H}_n=I^H_n(\mathcal H^{\otimes n})$$
be the fractional Wiener chaos of order $n$.

\begin{Veta}\label{jzn}
Under the assumptions of Theorem~\ref{ex.lin.} the mild solution
$$\{X_t=S_B(B^H_t)U(t,0)x,\ t\in [0,T]\},$$
to the equation~{\rm\Ref{jzn.rce}} is unique in ${\rm
Dom}\,\delta_H$.
\end{Veta}

\begin{proof}
Clearly,
$X=\{X_t,t\in[0,T]\}\in\mathbb{D}^{1,2}_H(|\mathcal{H}|)\subset{\rm
Dom}\,\delta_H$. Take another mild solution
$Y=\{Y_t,t\in[0,T]\}\in {\rm Dom}\,\delta_H$ to the
equation~\Ref{jzn.rce}. Then the processes $X$ and $Y$ satisfy
\begin{align*}
X_t&=S_A(t)x+\int_0^t S_A(t-r)BX_r\d B^H_r,\\
Y_t&=S_A(t)x+\int_0^t S_A(t-r)BY_r\d B^H_r,
\end{align*}
respectively. Define the process $Z=\{Z_t,t\in[0,T]\}$ as
$$Z_t=X_t-Y_t,\ t\in[0,T].$$
Let
$$Z_t=\sum_{n=0}^{+\infty} I_n\big(z_n(t,\,.\,)\big)$$
be the fractional Wiener chaos decomposition of process $Z$,
where $z_n(t,\,.\,)\in\mathcal{H}^{n+1}$ be the symmetric elements
in the last $n$~variables. Since
$$z_0(t)=I_0\big(z_0(t)\big)=\E[Z_t]=\E[X_t-Y_t]=S_A(t)x-S_A(t)x=0$$
for all $t\in[0,T]$, we get
$$Z_t=\sum_{n=1}^{+\infty} I_n\big(z_n(t,\,.\,)\big).$$
The definition of Skorokhod integral via multiple integrals yields 
\begin{align*}
\sum_{n=1}^{+\infty} I_n\big(z_n(t,\,.\,)\big)&=Z_t=\int_0^t S_A(t-r)BZ_r\d B^H_r=\int_0^t\sum_{n=0}^{+\infty} I_n\big(S_A(t-r)B z_n(r,\,.\,)\big)\d B^H_r\\&=
\sum_{n=0}^{+\infty}I_{n+1}\big(\textrm{Sym}\big(S_A(t-\,.\,)B z_n(\,.\,)\big)\big)=
\sum_{n=1}^{+\infty}I_n\big(\textrm{Sym}\big(S_A(t-\,.\,)B z_{n-1}(\,.\,)\big)\big),
\end{align*}
where $\textrm{Sym}(f)$ denotes the symmetrization of $f$ in all variables. From the uniqueness of Wiener chaos expansion it follows
$$z_n(t,\,.\,)=\textrm{Sym}\big(S_A(t-\,.\,)B z_{n-1}(\,.\,)\big),\ n\geq 1.$$
Since $z_0\equiv 0$ we obtain by induction that
$$z_1\equiv 0,z_2\equiv 0,\ldots$$
hence $Z\equiv 0$ and the proof is completed.
\end{proof}

\section{Uniqueness of weak solution}\label{jzn_slab_res}
Let $\mathcal M$ be the space of $\mathcal{B}([0,T])\otimes\mathcal F$--measurable processes $Z:[0,T]\times\Omega\rightarrow V$ with continuous trajectories such that
$$Z\in {\rm Dom}\,\delta_H\quad {\rm and}\quad \E\sup_{t\in[0,T]}\|Z_t\|^2_V<+\infty.$$

The aim is to show that a weak solution to the equation
\begin{equation}\label{rce_w}
\d X_t = \big(AX_t + F(t)\big)\d t + BX_t\d B^H_t,\ X_0=x,
\end{equation}
is unique in the space $\mathcal M$. On this purpose the following
version of integration by parts formula is necessary.

\begin{Lemma}\label{per_partes}
Let $Y\in\mathcal M$ be a weak solution to the equation
$$\d Y_t = AY_t\d t + BY_t\d B^H_t,\ X_0=0.$$
Then
\begin{equation}\label{p_p}
\big\langle Y_t,\zeta(t)\big\rangle_V = \int_0^t \big\langle
Y_r,A^*\zeta(r)+\zeta'(r)\big\rangle_V \d r + \int_0^t
\big\langle Y_r,B^*\zeta(r)\big\rangle_V\d B^H_r
\end{equation}
for any $\zeta\in\mathcal{C}^1\big([0,T];D^*\big)$.
\end{Lemma}

\begin{proof}
{\bf 1st step:} Let $\zeta$ have the form
\begin{equation}\label{jedn_fce}
\zeta(t)=\varphi(t)\xi,\ \varphi\in\mathcal{C}^1([0,T]),\xi\in
D^*.
\end{equation}
Let $\{t_k,k=0,\ldots,n\}$ be the partition of interval $[0,t]$.
Then
\begin{align}
\langle Y_t,\zeta(t)\rangle_V &= \varphi(t)\langle Y_t,\xi\rangle_V = \sum_{k=0}^{n-1} \big(\varphi(t_{k+1})\langle Y_{t_{k+1}},\xi\rangle_V - \varphi(t_k)\langle Y_{t_k},\xi\rangle_V\big)\label{rozpis}\\& = \sum_{k=0}^{n-1} \big(\varphi(t_{k+1})-\varphi(t_k)\big)\langle Y_{t_{k+1}},\xi\rangle_V + \sum_{k=0}^{n-1} \varphi(t_k)\big(\langle Y_{t_{k+1}},\xi\rangle_V - \langle Y_{t_k},\xi\rangle_V\big)=S_1+S_2.\nonumber
\end{align}
Since
$$
\left|\sum_{k=0}^{n-1} \big(\varphi(t_{k+1})-\varphi(t_k)\big)\langle Y_{t_{k+1}},\xi\rangle_V\right|\leq \|\varphi\|_{\mathcal{C}^1([0,T])}\sup_{r\in[0,T]}\|Y_r\|_V \left(\sum_{k=0}^{n-1} (t_{k+1}-t_k)\right)\|\xi\|_V$$
and
$$\E\Big[\sup_{r\in[0,T]}\|Y_r\|^2_V\Big]<+\infty,$$
it follows
$$S_1=\sum_{k=0}^{n-1} \big(\varphi(t_{k+1})-\varphi(t_k)\big)\langle Y_{t_{k+1}},\xi\rangle_V\konv{n} \int_0^t \varphi'(r)\big\langle Y_r,\xi\big\rangle_V\d r\quad \textrm{in}\ L^2(\Omega)$$
in virtue of the Lebesgue dominated convergence theorem.
\noindent
The second sum $S_2$ can be split into two summands
$$S_2=\sum_{k=0}^{n-1} \varphi(t_k)\left(\int_{t_k}^{t_{k+1}}\langle Y_r,A^*\xi\rangle_V\d r + \int_{t_k}^{t_{k+1}}\langle Y_r,B^*\xi\rangle_V\d B^H_r\right)=S_{21}+S_{22}.$$
The first summand
$$S_{21}=\sum_{k=0}^{n-1} \varphi(t_k)\int_{t_k}^{t_{k+1}}\langle Y_r,A^*\xi\rangle_V\d r\konv{n} \int_0^t \varphi(r)\big\langle Y_r,A^*\xi\big\rangle_V\d r\quad \textrm{in}\ L^2(\Omega)$$
by the Lebesgue dominated convergence theorem because
$$\left|\sum_{k=0}^{n-1} \varphi(t_k)\int_{t_k}^{t_{k+1}}\langle Y_r,A^*\xi\rangle_V\d r\right|\leq \|\varphi\|_{\mathcal{C}^1([0,T])}\sup_{r\in[0,T]}\|Y_r\|_V\|A^*\xi\|_V T.$$
Since $Y_t\in L^2(\Omega)$ satisfies \Ref{rozpis} we conclude that
$$S_{22}=\sum_{k=0}^{n-1} \varphi(t_k)\int_{t_k}^{t_{k+1}}\langle Y_r,B^*\xi\rangle_V\d B^H_r = \int_0^t \sum_{k=0}^{n-1} \varphi(t_k)I_{(t_k,t_{k+1}]}(r)\langle Y_r,B^*\xi\rangle_V\d B^H_r$$
converges to a random variable denoted by $Y_t^1$ in $L^2(\Omega)$ as $n\rightarrow +\infty$.
\noindent
It remains to show that $Y^1_t=\int_0^t \varphi(r)\big\langle
Y_r,B^*\xi\big\rangle_V\d B^H_r$ by the closedness of Skorokhod
integral. Denote
$$\Phi_n(r)=\sum_{k=0}^{n-1} \varphi(t_k)I_{(t_k,t_{k+1}]}(r)\langle Y_r,B^*\xi\rangle_V,\ r\in[0,t], n\in\N.$$
Then $\Phi_n\in\rm{Dom}\,\delta_H$,
$$\Phi_n(r)\konv{n} \varphi(r)\langle Y_r,B^*\xi\rangle_V$$
for any fixed $r,\omega$, and
\begin{align*}
\big|\Phi_n(r)\big|&=\Big|\sum_{k=0}^{n-1} \varphi(t_k)I_{(t_k,t_{k+1}]}(r)\langle Y_r,B^*\xi\rangle_V\Big| \leq
\|\varphi\|_{\mathcal{C}^1([0,T])}\sup_{r\in[0,T]}\|Y_r\|_V\|B^*\xi\|_V\sum_{k=0}^{n-1} I_{(t_k,t_{k+1}]}(r)\\&=
\|\varphi\|_{\mathcal{C}^1([0,T])}\sup_{r\in[0,T]}\|Y_r\|_V\|B^*\xi\|_V.
\end{align*}
By the Lebesgue dominated convergence theorem $\Phi_n\in L^2\big(\Omega;L^2([0,t];V)\big)$ and
$$\Phi_n\konv{n} \varphi \langle Y,B^*\xi\rangle_V\quad {\rm in}\quad L^2\big(\Omega;L^2([0,t];V)\big).$$
By the closedness of Skorokhod integral $Y^1_t=\int_0^t
\varphi(r)\big\langle Y_r,B^*\xi\big\rangle_V\d B^H_r$ and
equality~\Ref{p_p} holds for $\zeta$ of the form~\Ref{jedn_fce}.\\

\noindent {\bf 2nd step:} Let $\zeta\in
\mathcal{C}^1\big([0,T];D^*\big)$. Then there exists a sequence
$\{\zeta_n,n\in \N\}\subset \mathcal{C}^1\big([0,T];D^*\big)$ of
elementary functions of the form~\Ref{jedn_fce} such that
$\zeta_n\konv{n} \zeta$ in $\mathcal{C}^1\big([0,T];D^*\big)$.
The aim is to pass to the limit in the equation
$$\big\langle Y_t,\zeta_n(t)\big\rangle_V = \int_0^t \big\langle
Y_r,A^*\zeta_n(r)+\zeta'(r)\big\rangle_V \d r + \int_0^t
\big\langle Y_r,B^*\zeta_n(r)\big\rangle_V\d B^H_r$$ in
$L^2(\Omega)$. Clearly,
$$\big|\big\langle Y_t,\zeta_n(t)-\zeta(t)\big\rangle_V\big|\leq \sup_{r\in[0,T]}\|Y_r\|_V \|\zeta_n-\zeta\|_{\mathcal{C}^1([0,T];D^*)}\konv{n} 0$$
and
\begin{align*}
\E&\Big[\Big(\int_0^t \big\langle
Y_r,A^*\big(\zeta_n(r)-\zeta(r)\big)+\big(\zeta'_n(r)-\zeta'(r)\big)\big\rangle_V\d
r\Big)^2\Big]\\&\qquad\leq \E\Big[\sup_{r\in[0,T]}\|Y_r\|^2_V\Big]
\|\zeta_n-\zeta\|^2_{\mathcal{C}^1([0,T];D^*)}T^2\konv{n} 0,
\end{align*}
thus $\int_0^t \big\langle Y_r,B^*\zeta_n(r)\big\rangle_V\d
B^H_r\konv{n} Y^2_t\quad \textrm{in}\ L^2(\Omega)$. By the
closedness of Skorokhod integral, $Y^2_t=\int_0^t \big\langle
Y_r,B^*\zeta(r)\big\rangle_V\d B^H_r$ because
$$\E\Big[\int_0^t \!\big\langle
Y_r,B^*\big(\zeta_n(r)-\zeta(r)\big)\big\rangle^2_V\d r\Big]\leq
\E\Big[\sup_{r\in[0,T]}\!\|Y_r\|^2_V\Big]\|B^*\|^2_{\mathcal{L}(V)}
\|\zeta_n-\zeta\|^2_{\mathcal{C}^1([0,T];D^*)}T\konv{n} 0$$
and $\big\langle Y,B^*\zeta_n\big\rangle_V\in \rm{Dom}\,\delta_H$
for any $n\in\N$.
\end{proof}

\noindent
Now, we are able to prove the uniqueness result.

\begin{Veta}\label{w_jzn}
Under the assumptions of Theorem~\ref{main.nelin} the solution to
the equation~{\rm \Ref{rce_w}} is unique in the space~$\mathcal
M$.
\end{Veta}

\begin{proof}
Let $X^M$ be the solution to the equation
$$X^M_t=U_Y(t,0)x + \int_0^tU_Y(t,r)F(r)\d r,\ t\in [0,T],$$
which is also a weak one to \Ref{rce_w} (Theorem~\ref{main.nelin}), where
$$U_Y(t,s)=S_B(B^H_t-B^H_s)U(t-s,0),\ s\leq t\leq T,$$
(for more details see~\Ref{ns.res}). Using the notation from
Section~\ref{kap.2}
$$\|U_Y(t,s)\|_{\mathcal{L}(V)}\leq M_B{\rm exp}\left\{2\omega_B\|B^H\|_{\mathcal{C}([0,T])}\right\}C_U,\ 0\leq s\leq t\leq T,$$
it follows
\begin{align*}
\E\Big[\sup_{t\in[0,T]}\|X^M_t\|^2_V\Big]&\leq
2C^2_UM^2_B\Big(\|x\|^2_V\E{\rm
exp}\left\{2\omega_B\|B^H\|_{\mathcal{C}([0,T])}\right\}
\\&\qquad\qquad+
\|F\|^2_{L^2([0,T])}T\E{\rm
exp}\left\{4\omega_B\|B^H\|_{\mathcal{C}([0,T])}\right\}\Big)<+\infty,
\end{align*}
by the Fernique Theorem. Therefore, $X^M\in\mathcal M$ (the continuity of trajectories is guaranteed by Theorem~\ref{mild.res}).

\noindent
Take another weak solution $X^1\in\mathcal M$ to \Ref{rce_w} and define
$$\bar X:=X^1-X^M.$$
Then $\bar X$ is a weak solution to the equation
$$\d \bar X_t = A\bar X_t\d t + B\bar X_t\d B^H_t,\ \bar X_0=0.$$
Hence, applying Lemma~\ref{per_partes} to $\langle \bar X_t,\xi\rangle_V$ for any fixed $\xi \in D^*$ and $\zeta(s)=S_A^*(t-s)\xi,s\in[0,t],$ it follows
\begin{align*}
\langle \bar X_t,\xi\rangle_V &= \int_0^t \big\langle
\bar X_r,A^*S_A^*(t-r)\xi-S_A^*(t-r)A^*\xi\big\rangle_V \d r + \int_0^t
\big\langle \bar X_r,B^*S_A^*(t-r)\xi\big\rangle_V\d B^H_r\\&= \int_0^t
\big\langle S_A(t-r)B\bar X_r,\xi\big\rangle_V\d B^H_r= \Big\langle\int_0^t S_A(t-r)B\bar X_r\d B^H_r,\xi\Big\rangle_V.
\end{align*}
Thus Theorem~\ref{mild.res} yields
$$\bar X_t=\int_0^t S_A(t-r)B\bar X_r\d B^H_r=S_B(B^H_t)U(t,0)0=0$$
and $X^1=X^M$.
\end{proof}

\begin{Dusledek}\label{dusl}
The weak solution $\big\{S_B(B^H_t)U(t,0)x,t\in[0,T]\big\}$ to the equation~{\rm\Ref{jzn.rce}} is unique in $\mathcal M$.

\noindent
In particular, the solution
$$X_t=\exp\left\{bB^H_t-\frac{1}{2}b^2t^{2H}+at\right\}x,\ t\in[0,T],$$
to the one--dimensional equation
$$\d X_t = a X_t\d t + b X_t\d B^H_t,\ X_0 = x,$$
is unique in $\mathcal M$.
\end{Dusledek}

\par
\noindent \textsc{Acknowledgements:} \it{The authors would like to
thank M.~Jolis for her help with Section~\ref{jzn_mild_res} and J.~Seidler for his valuable suggestions and comments.}

\bibliographystyle{spmpsci}
\bibliography{references}

\begin{thebibliography}{10}
\providecommand{\url}[1]{{#1}}
\providecommand{\urlprefix}{URL }
\expandafter\ifx\csname urlstyle\endcsname\relax
  \providecommand{\doi}[1]{DOI~\discretionary{}{}{}#1}\else
  \providecommand{\doi}{DOI~\discretionary{}{}{}\begingroup
  \urlstyle{rm}\Url}\fi

\bibitem{alos}
Al{\`{o}}s, E., Nualart, D.: Stochastic integration with respect to the
  fractional {B}rownian motion.
\newblock Stoch. Stoch. Rep. \textbf{75}, 129--152 (2003)

\bibitem{bartek}
B{\'a}rtek, J., Garrido-Atienza, M., Maslowski, B.: Random dynamical systems
  defined by fractional stochastic porous media equations.
\newblock Stoch. Dyn. \textbf{13}, 1--33 (2013)

\bibitem{bonaccorsi}
Bonaccorsi, S.: Nonlinear stochastic differential equations in infinite
  dimensions.
\newblock Stoch. Anal. Appl. \textbf{18}(3), 333--345 (2000)

\bibitem{daprato2}
{{Da} Prato}, G., Iannelli, M., Tubaro, L.: An existence result for a linear
  abstract stochastic equation in hilbert spaces.
\newblock Rend. Sem. Mat. Univ. Padova \textbf{67}, 171--180 (1982)

\bibitem{daprato1}
{{Da} Prato}, G., Iannelli, M., Tubaro, L.: Some results on linear stochastic
  differential equations in hilbert spaces.
\newblock Stochastics \textbf{6}(2), 105--116 (1982)

\bibitem{daprato}
{{Da} Prato}, G., Zabczyk, J.: Stochastic {E}quations in {I}nfinite
  {D}imensions.
\newblock Cambridge University Press, Cambridge (1992)

\bibitem{duncan}
Duncan, T.E., Maslowski, B., Pasik-Duncan, B.: Stochastic equations in
  {H}ilbert space with a multiplicative fractional gaussian noise.
\newblock Stochastic Process. Appl. \textbf{115}, 1357--1383 (2005)

\bibitem{fernique}
Fernique, X.: R{\'e}gularit{\'e} des trajectoires des fonctions al{\'e}atoires
  gaussiennes.
\newblock {\'E}cole d'{\'E}t{\'e} de Probabilit{\'e}s de Saint--Flour IV-1974,
  LNM 480, Springer--Verlag, Berlin pp. 1--96 (1975)

\bibitem{flandoli2}
Flandoli, F.: On the semigroup approach to stochastic evolution.
\newblock Stochastic Anal. Appl. \textbf{10}(2), 181--203 (1992)

\bibitem{GMS}
Garrido-Atienza, M., Maslowski, B., {\v S}nup{\' a}rkov{\' a}, J.: Semilinear
  stochastic equations with bilinear fractional noise.
\newblock Discrete Contin. Dyn. Syst., Ser. B \textbf{21}(9), 3075--3094 (2016)

\bibitem{leon}
Le{\'o}n, J.A., Nualart, D.: Stochastic evolution equations with random
  generators.
\newblock Ann. Probab. \textbf{26}, 149--186 (1998)

\bibitem{nourdin}
Noudrin, I., Peccati, G.: Normal {A}pproximations with {M}alliavin {C}alculus:
  {F}rom {S}tein's {M}ethod to {U}niversality.
\newblock University Press, Cambridge (2012)

\bibitem{nualart}
Nualart, D.: The {M}alliavin calculus and related topics.
\newblock Springer-Verlag, New York (1995)

\bibitem{nualart1}
Nualart, D.: Stochastic integration with respect to fractional {B}rownian
  motion and applications.
\newblock Stochastic models (Mexico City, 2002), Contemp. Math.,Amer. Math.
  Soc., Providence, RI \textbf{336}, 3--39 (2003)

\bibitem{pazy}
Pazy, A.: Semigroups of {L}inear {O}perators and {A}pplications to {P}artial
  {D}ifferential {E}quations.
\newblock Springer--Verlag, New York (1983)

\bibitem{perez-abreu}
P{\'e}rez-Abreu, V., Tudor, C.: Multiple stochastic fractional integrals: {A}
  transfer principle for multiple stochastic fractional integrals.
\newblock Bol. Soc. Mex. Matem. \textbf{8}(3), 187--203 (2002)

\bibitem{snuparkova}
{\v S}nup{\' a}rkov{\' a}, J.: Stochastic bilinear equations with fractional
  {G}aussian noise in {H}ilbert space.
\newblock Acta Univ. Carolin. Math. Phys. \textbf{51}, 49--68 (2010)

\bibitem{tanabe}
Tanabe, H.: Equations of {E}volutions.
\newblock Pitman, London (1979)

\end{thebibliography}

\end{document}